\documentclass[11pt,reqno,a4paper]{amsart}

\usepackage{amsmath, amsthm, amssymb, amsfonts}
\usepackage{setspace}
\usepackage[english]{babel}
\usepackage{mathrsfs}
\usepackage[all]{xy}
\usepackage{paralist}
\usepackage{graphicx}
\usepackage{accents}
\usepackage{enumerate}
\usepackage{mathtools}

\usepackage[dvipsnames]{xcolor}
\usepackage[colorlinks=true,linkcolor=OrangeRed,urlcolor=MidnightBlue,citecolor=OrangeRed]{hyperref}

\usepackage[margin=2.5cm,top=2.5cm,bottom=2cm]{geometry}

\makeatletter
\newtheorem*{rep@theorem}{\rep@title}
\newcommand{\newreptheorem}[2]{%
\newenvironment{rep#1}[1]{%
 \def\rep@title{#2~\ref{##1}}%
 \begin{rep@theorem}}%
 {\end{rep@theorem}}}
 
\makeatother

\theoremstyle{plain}
\newtheorem{theorem}{\normalfont \scshape Theorem}[section]
\newtheorem{lemma}[theorem]{\normalfont \scshape Lemma}
\newtheorem{proposition}[theorem]{\normalfont \scshape Proposition}
\newtheorem{corollary}[theorem]{\normalfont \scshape Corollary}

\newreptheorem{theorem}{Theorem}
\newreptheorem{proposition}{Proposition}

\theoremstyle{definition}

\theoremstyle{remark}
\newtheorem{remark}[theorem]{\normalfont \scshape Remark}

\renewenvironment{proof}[1][\proofname]{\noindent{\scshape #1}.\quad}{\qed}

\numberwithin{equation}{section}

\def\XXint#1#2#3{{\setbox0=\hbox{$#1{#2#3}{\int}$ }
\vcenter{\hbox{$#2#3$ }}\kern-.6\wd0}}

\newcommand{\bbI}{\mathbb I}

\newcommand{\bbR}{\mathbb R}

\newcommand{\calC}{\mathcal C}
\newcommand{\calD}{\mathcal D}

\newcommand{\calH}{\mathcal H}
\newcommand{\calI}{\mathcal I}
\newcommand{\calJ}{\mathcal J}

\newcommand{\calP}{\mathcal P}

\newcommand{\bp}{\begin{pmatrix}}
\newcommand{\ep}{\end{pmatrix}}
\newcommand{\p}{\partial}

\newcommand{\R}{\mathbb{R}}

\newcommand{\T}{\mathbb{T}}

\newcommand{\dx}{\textnormal{d}x}

\newcommand{\dv}{\textnormal{d}v}

\newcommand{\ddt}{\frac{\textnormal{d}}{\textnormal{d}t}}

\newcommand{\lt}{\left}
\newcommand{\rt}{\right}
\newcommand{\pa}{\partial}

\newcommand{\iinttr}{\iint_{\T^d\times\R^d}}
\newcommand{\intt}{\int_{\T^d}}
\newcommand{\intr}{\int_{\R^d}}

\newcommand{\e}{\varepsilon}

\date{\today}

\begin{document}

\title[Conditional hypocoercivity]{Conditional hypocoercivity for  nonlinear kinetic Fokker--Planck equations}

\author{Jos\'e A. Carrillo}
\address{Mathematical Institute, University of Oxford,  Oxford OX2 6GG, United Kingdom}
\email{jose.carrillo@maths.ox.ac.uk}

\author{Dowan Koo}
\address{Mathematical Institute, University of Oxford,  Oxford OX2 6GG, United Kingdom}
\email{dowan.koo@maths.ox.ac.uk}

\author{Sihyun Song}
\address{Department of Mathematics, Yonsei University, Seoul 03722, Republic of Korea}
\email{ssong@yonsei.ac.kr}

\date{\today}

\thanks{JAC was supported by the Advanced Grant Nonlocal-CPD (Nonlocal PDEs for Complex Particle Dynamics: Phase Transitions, Patterns and Synchronization) of the European Research Council Executive Agency (ERC) under the European Union Horizon 2020 research and innovation programme (grant agreement No. 883363) and partially supported by the EPSRC EP/V051121/1. JAC was partially supported by the ``Maria de Maeztu'' Excellence Unit IMAG, reference CEX2020-001105-M, funded by MCIN\slash AEI \slash10.13039\slash501100011033\slash.  DK was supported by NRF grant no. RS-2025-02312778. SS was  supported by NRF grants no. 2022R1A2C1002820 and RS-2024-00406821. SS acknowledges support from the Presidential Science Scholarship provided by the Korea Student Aid Foundation.}

\begin{abstract}
We investigate the long-time behaviour of nonlinear kinetic Fokker--Planck equations with porous medium diffusion in a non-perturbative setting. Under a priori conditional bounds on macroscopic quantities, we establish exponential convergence to equilibrium in $L^1$. These bounds are automatically satisfied if the initial data is trapped between two global equilibrium profiles. Our approach combines the entropy–entropy dissipation structure and some techniques from $L^2$-hypocoercivity.
\end{abstract}

\maketitle
\setcounter{tocdepth}{1}

\singlespacing

\section{Introduction}

In this paper we are interested in establishing the convergence to equilibrium of solutions to the following nonlinear kinetic Fokker--Planck equations:
\begin{equation}\label{eq: main}
    (\pa_t  + v\cdot \nabla_x) f = \Delta_v f^m + \nabla_v\cdot(vf) \quad \text{on} \quad \R_+\times \T^d\times\R^d, \quad m\ge 1, 
\end{equation}
with particular emphasis in the nonlinear porous medium regime $m>1$.

For the spatially homogeneous counterpart of \eqref{eq: main},
\begin{equation}\label{eq: par}
    \pa_t f = \Delta_v f^m + \nabla_v\cdot(vf),
\end{equation}
the long-time behaviour of solutions has been extensively studied. The entropy functional for \eqref{eq: par} was provided in \cite{N84,R84} together with compactness arguments leading to convergence to the equilibrium state without rates. A nonlinear version of the Bakry--Emery approach was developed in \cite{CT00,CJMTU01} to establish trend to equilibrium.  The long-time behaviour of \eqref{eq: par} can be understood by using generalized Sobolev/Gagliardo-Nirenberg inequalities \cite{DPD02}, as classically done in the case $m=1$ by log-Sobolev inequalities. These nonlinear versions of the log-Sobolev inequality were further generalized for a large family of pressure laws in \cite{CJMTU01}, see \cite{ACM26} for the latest developments. We also refer to \cite{O01} concerning the important gradient flow structure of \eqref{eq: par} that also leads to the understanding of rates of convergence and hidden convexity structures. 

Turning to the spatially inhomogeneous case of \eqref{eq: main}, some developments have recently appeared. In \cite{BCD26} the fundamental solution was discussed (see also the paragraphs below) for a certain range of $m$. More recently, in \cite{BDM26+,BCDQ26+}, the gradient flow structure, existence of weak solutions, and diffusion limit of \eqref{eq: main} were investigated. In \cite{CK26+},  the hydrodynamic limit from a variant of \eqref{eq: main} toward isentropic Euler flow has been established. 

However, it remains open whether a solution to \eqref{eq: main} converges in the long-time to the global equilibrium state. Therefore, we are interested in establishing a hypocoercivity theory for \eqref{eq: main}. We refer to the monograph \cite{V09} and the excellent lecture notes \cite{H18}, along with the references therein, for an introduction to this topic. Specifically to the linear case of \eqref{eq: main} (namely $m=1$), we may also refer to \cite{MM16,GMS17}, as well as \cite{CM24} where general domains with Maxwell boundary conditions were incorporated in the spirit of \cite{BCMT23}.

Although the theory of hypocoercivity is quite rich, an extension to the model \eqref{eq: main} appears nontrivial and has not been developed. Yet in the aforementioned \cite{BCD26}, an explicit computation for the fundamental solution of the equation \eqref{eq: main} without drift, namely
\begin{equation*}
    \p_t f + v\cdot \nabla_x f = \Delta_v f^m,
\end{equation*}
was provided, given that $m\in (1-1/d,1)\cup (1,1+1/d)$. Roughly speaking, the main result of \cite{BCD26} states that if the initial data is trapped between two stationary Barenblatt profiles, then the long-time behaviour of the solution $f$ is governed by the fundamental solution $f_\star$, in the sense that $\|f(t,\cdot,\cdot)-f_\star(t,\cdot,\cdot)\|_{L^1_{x,v}} \to 0$ as $t\to\infty$. 

Prior to stating our main result, we remind the reader of the global equilibria corresponding to \eqref{eq: main}. For $m=1$, it is the Maxwellian $M_1(v) := e^{-|v|^2/2}/(2\pi)^{d/2}$. When $m>1$, it is the Barenblatt profile
\begin{equation*}
    M_m(v) := \lt( \frac{m-1}{m} \lt( c_{m,d} - \frac{|v|^2}{2} \rt) \rt)_+^{\frac{1}{m-1}},
\end{equation*}
Here we fix $c_{m,d}>0$ so that $M_m$ is a probability measure on $\bbR^d$. More generally, if $\rho>0$ is a given density number, then
\begin{equation}\label{lmm}
    M_{m,\rho}(v) := \lt( \frac{m-1}{m} \lt( c_{m,d} \rho^{\gamma-1} - \frac{|v|^2}{2} \rt) \rt)_+^{\frac{1}{m-1}} \quad \text{with} \quad \gamma:= 1 + \lt(\frac{d}{2}+ \frac{1}{m-1}\rt)^{-1} 
\end{equation}
satisfies $\|M_{m,\rho}\|_{L^1(\R^d)} = \rho$. We refer to $M_{m,\rho}$ as the local Barenblatt profile corresponding to $\rho$, and we remark that this one is deeply related to barotropic gas dynamics owing to the form of its pressure tensor \cite{B99,BV05,CK26+}.

Our main result concerns \textit{conditional hypocoercivity} for the nonlinear kinetic Fokker--Planck equations in \eqref{eq: main}. The conditional bounds that we impose on the solution are automatically satisfied for a wide class of initial data; see Remark \ref{rem:pizza} and Corollary \ref{cor: main} below. The conditions we impose are that the solution $f$ to \eqref{eq: main} satisfies
\begin{equation}\label{paella} 
\begin{aligned}
    &0< \lambda \le \rho_f(t,x) \le \Lambda \quad \text{where} \quad \rho_f := \intr f \,\dv,  \\
    &\intr |v|^2f(t,x,v)\,\dv \le E,\quad \intr f^m(t,x,v)\,\dv \le \Lambda_m.
\end{aligned}
\end{equation}
We point out that these bounds are imposed only on the {macroscopic quantities} of the solution. These conditions are reminiscent of the ones used in the program of conditional regularity for collisional kinetic equations, see for instance \cite{S16,HS20,IS20,S20}. The quantity $\intr f^m\,\dv$ plays the role of the Boltzmann entropy $\intr f\log f \,\dv$ when $m>1$. However, we do not need an upper bound assumption on the Boltzmann entropy in the case where $m=1$.

\begin{theorem}\label{thm: main}
    Let $m\in [1,\infty)$ be given and $f$ be a global-in-time weak solution to \eqref{eq: main} with unit mass. We further assume that there exists a finite $T_0\ge 0$ such that \eqref{paella} holds for all $(t,x)\in [T_0,\infty)\times \T^d$. Then for all $t \ge T_0$,
    \[
    \|f(t) - M_m\|_{L^1(\T^d\times \R^d)} \le Ce^{-\nu (t-T_0)},
    \]
    where $C=C(m,d)>0$, $\nu=\nu(\lambda,\Lambda,E,\Lambda_m,m,d)>0$.
\end{theorem}

\begin{remark} \label{rem:pizza}
    As aforementioned, there is a wide class of initial data such that \eqref{paella} hold automatically for the corresponding solutions to \eqref{eq: main}. Namely, if it is the case that the initial data satisfies the pointwise bounds
    \begin{equation}\label{eq:truffle}
        \tilde\Lambda^{-1} M_m \le f(0) \le \tilde\Lambda M_m \quad \forall (x,v)\in \T^d\times\R^d
    \end{equation}
    for some $\tilde \Lambda>0$,
    then these bounds propagate with time, by virtue of the maximum principle. This implies that the solutions whose initial data satisfy \eqref{eq:truffle} verify the conditional bounds \eqref{paella} with $T_0=0$, so that the conclusion of Theorem \ref{thm: main} automatically holds:

\begin{corollary}\label{cor: main}
    We let $m\in [1,\infty)$ be given, and we assume that the initial data $f(0)$ satisfies \eqref{eq:truffle} for some $\tilde\Lambda > 0$. Suppose also that $f(0)$ has unit mass. Then for the corresponding global-in-time weak solution to \eqref{eq: main}, there exist $C,\nu>0$ (depending on $\tilde\Lambda, m, d$) such that for every $t\ge 0$, there holds
    \begin{equation*}
        \|f(t) - M_m\|_{L^1(\T^d\times\R^d)} \le C e^{-\nu t}.
    \end{equation*}
\end{corollary}

Initial data satisfying pointwise bounds such as \eqref{eq:truffle} were considered in \cite{DV01,NS05,FPS23,AZ24,EMY26,APT26,PT26} to study hypocoercivity in various kinetic models. 
    \end{remark}
    
    \begin{remark}
     In \cite[Theorem 1.1]{DV01}, the trend to equilibrium of solutions to \eqref{eq: main} for $m=1$, in the whole space with a spatially confining potential, was studied for initial data trapped between two global equilibria. 
     
     The results of \cite{DV01} have been improved, in several directions, throughout the literature. For instance, in \cite[Theorem 38]{V09}, the class of initial data is relaxed to those having arbitrarily large moments. This assumption allows the
employment of hypoelliptic estimates, which are in turn necessary because Villani's proof uses a modified entropy functional which incorporates the Fisher information in both the space and velocity variables. In the case of $m>1$, this approach appears out of reach because it is not known whether such a gain of regularity is possible. Here, we bypass this difficulty because our methods do not need incorporation of any $H^1$-type entropy.
\end{remark}

\begin{remark}
    We would like to acknowledge that we drew inspiration from the recent preprint \cite{APT26}, as well as \cite{NS05,FPS23,PT26}. In \cite{APT26}, the hypocoercivity of a quantum Fokker--Planck equation was studied for initial data satisfying bounds of the form \eqref{eq:truffle}, these of which propagate with time. Using the lower and upper pointwise bounds on $f$, the equivalence between the relative entropy and a weighted $L^2$ norm was established. Then, $L^2$-hypocoercivity methods \cite{DMS15} were used to close the estimates.

    A crucial different approach we take is that instead of using pointwise estimates on $f$ such as \eqref{eq:truffle} to convert to an $L^2$-framework, we stay at the framework of entropy-entropy dissipation. This allows us to obtain coercive estimates based on only macroscopic bounds such as \eqref{paella}.
    Finally, we emphasize that our trend to equilibrium result holds in the  $L^1$ space without any weight.
    \end{remark}

\begin{remark}
    Using our arguments with entropy, we can establish the long-time behaviour of certain other nonlinear kinetic Fokker--Planck equations as well, for instance
    \begin{equation*}
        \p_t f + v\cdot \nabla_x f = \rho_f^\beta (\Delta_v f + \nabla_v\cdot (vf)), \quad \beta>0.
    \end{equation*}
    Our approach therefore provides an alternative perspective to \cite{AZ24}, who worked in the $L^2$-framework.
\end{remark}

We briefly explain the idea behind Theorem \ref{thm: main}. The distance between $f$ and the global equilibrium $M_m$ is measured using a relative entropy with respect to $M_m$. By the entropy-entropy dissipation structure of \eqref{eq: main} and logarithmic Sobolev-type inequalities \cite{DPD02,CJMTU01}, we first obtain dissipation of the relative entropy, but only with respect to the local equilibrium $M_{m,\rho_f}$:
\[
\ddt \calH_m[f|M_m]\lesssim -\calH_m[f|M_{m,\rho_f}].
\]
This inequality does not provide full dissipation toward the global equilibrium. It is mainly due to the fact that the dissipation fails for any local equilibria. To capture the relaxation of the local equilibrium $M_{m,\rho_f}$ towards the global one $M_{m}$, we construct a Lyapunov functional by incorporating a corrector term which is well-known in the theory of hypocoercivity \cite{DMS09,G10,EGKM13,DMS15,BG16,APT26}. This corrector term is typically adopted towards the end of closing estimates in the weighted $L^2$ framework. However, under the conditional bounds of \eqref{paella}, we prove that it is also efficient in closing estimates in the relative entropy framework (see for instance Lemma \ref{lem: comp} and Proposition \ref{prop: corrector}). We are then able to conclude to Theorem \ref{thm: main} thanks to Csisz\'ar--Kullback-type inequalities.

We finally point out that our main focus relies on the a priori estimates that yield, at least formally, the long-time behaviour of solutions. We do not investigate the development of a well-posedness theory of weak solutions satisfying the needed control on macroscopic quantities leading to our long-time behavior results being not conditional, as well as, the approximation of these weak solutions by smooth enough solutions for which the computations in this paper are fully rigorous.

The rest of the article is organized as follows. In Section \ref{sec: lin}, we use the linear case $m=1$ to showcase our approach. The key result is encapsulated in  Proposition \ref{prop: linFP}. In Section \ref{sec: nonlin}, we establish Proposition \ref{prop: corrector}, which is the nonlinear analogue of Proposition \ref{prop: linFP} for $m>1$. Finally, Section \ref{sec: thm} concludes the proof of Theorem \ref{thm: main}.



\section{The linear case $m=1$} \label{sec: lin}
To illustrate our approach, we first study the linear case of $m=1$
\begin{equation}\label{eq: linFP}
 (\pa_t + v\cdot \nabla_x)f = \Delta_v f + \nabla_v \cdot (vf),
\end{equation}
that will guide us for $m>1$ case in Section \ref{sec: nonlin}.

\subsection{Set-up}
Let us recall that the entropy associated to \eqref{eq: linFP} is written
\[
H_1[f]= \intr \lt(\frac{1}{2}|v|^2 +  \log f \rt) f\,\dv.
\]
The dissipation of the entropy
reads
\[
D_1[f]=\intr \frac{|\nabla_v f + vf|^2}{f}\,\dv.
\]
For $(\rho,U,T)\in \R_+\times \R^d \times \R_+^*$, we denote the local Maxwellian as
\[
M^{(\rho,U,T)}:= \frac{\rho}{(2\pi T)^{\frac{d}{2}}}e^{-\frac{|v-U|^2}{2T}},
\]
whereas the global Maxwellian is defined as
\[
M_1(v):= M^{(1,0,1)} = \frac{1}{(2\pi)^{\frac{d}{2}}}e^{-\frac{|v|^2}{2}}.
\]
The relative entropy for $m=1$ case is defined as the Kullback-Leibler divergence
\[
H_1[f|g]= \intr f \log \frac{f}{g}\,\dv.
\]
In particular, we get
\[
\begin{aligned}
H_1[f|M_1] = \intr f \log \frac{f}{M_1}\,\dv &= \intr \frac{1}{2}|v|^2 f + f\log f +\frac{d}{2}\log(2\pi) f\,\dv \\
&= H_1[f] + \frac{d}{2}\log(2\pi)\rho_f.
\end{aligned}
\]
Another important distribution is the stationary local Maxwellian
\begin{equation*}
    M_{1,\rho_f} := M^{(\rho_f,0,1)}=\frac{\rho_f}{(2\pi)^{\frac{d}{2}}}e^{-\frac{|v|^2}{2}}.
\end{equation*}
Thanks to the logarithmic Sobolev inequality (LSI), we have
\begin{equation}\label{eq:LSI1}
    D_1[f] \ge 2 H_1[f|M_{1,\rho_f}].
\end{equation}
Since $\intr f \,\dv = \intr M_{1,\rho_f}\,\dv =\rho_f$, Jensen's inequality shows
\[
H_1[f|M_{1,\rho_f}]= \intr \lt( f\log \frac{f}{M_{1,\rho_f}}- f + M_{1,\rho_f} \rt)\,\dv \ge 0.
\]
On the other hand, notice that
\[
H_1[f|M_1] - H_1[f|M_{1,\rho_f}] = \intr f \log \rho_f\,\dv = \rho_f \log \rho_f.
\]
Since $\intt \rho_f\,\dx =1$, Jensen's inequality again gives 
\[
\calP_1[\rho_f]:=\intt \rho_f \log \rho_f\,\dx \ge 0.
\]
Denoting
\[
\calH_1[\cdot] = \intt H_1[\cdot]\,\dx,\quad  \calD_1[\cdot] = \intt D_1[\cdot]\,\dx,
\]
we therefore also have
\begin{equation}\label{aux}
\calH_1[f|M_1] = \intt H_1[f|M_{1,\rho_f}]\,\dx + \calP_1[\rho_f]\ge 0.
\end{equation}
Then, the entropy--entropy dissipation equality associated to \eqref{eq: linFP} reads
\[
\ddt \calH_1[f(t)] = - \calD_1[f(t)].
\]
Thanks to mass conservation and the LSI \eqref{eq:LSI1}, we have
\[
\ddt \calH_1[f|M_1] = \ddt \calH_1[f] \le -\frac{1}{2}\calD_1[f(t)] -\calH_1[f|M_{1,\rho_f}].
\]
Obtaining the missing dissipation term $\calP_1[\rho_f]$ in \eqref{aux} is the key for the remaining subsections.

\subsection{The corrector term}
To obtain the decay of $\calP_1[\rho_f]$, we introduce the corrector term
\begin{equation}\label{coriolis}
\calC[f]:= \intt j_f \cdot \nabla_x\Phi_f\,\dx, \quad j_f:= \intr v f\,\dv,
\end{equation}
where $\Phi_f$ is the solution to the Poisson equation
\[
(-\Delta_x)\Phi_f = \rho_f -1,\quad \intt \Phi_f\,\dx =0.
\]
We shall soon see, in Lemma \ref{lem: m=1 comp}, that this crossing term may be bounded in terms of $\calH_1[f|M_1]$. This is possible due to the following equivalence:

\begin{lemma}\label{lem:baby1}
   If $\rho_f \ge 0$ satisfies $\intt \rho_f\,\dx =1$ and $\lambda \le \rho_f \le \Lambda$ for some $0< \lambda \le \Lambda <+\infty$, then it holds that
   \[
   \frac{1}{\Lambda} \|\rho_f-1\|_{L^2}^2\le\intt \rho_f \log \rho_f\,\dx \le \frac{1}{\lambda} \|\rho_f-1\|_{L^2}^2
   \]
   
\end{lemma}
\begin{proof}
    Let $\psi(x)= x \log x$ for $x>0$. For $x \in [\lambda,\Lambda]$, we have
    \[
    \psi(x) = (x-1) + \frac{1}{\tilde{x}}(x-1)^2
    \]
    for some $\tilde{x} \in [\lambda, \Lambda]$. This proves the desired inequality.
\end{proof}

\begin{lemma}\label{lem: m=1 comp}
    If $\rho_f \in L^1_+(\T^d)$ satisfies $\intt \rho_f \,\dx =1$ and $\rho_f \le \Lambda$, we have
    \[
   |\calC[f]| \le N^{(0)} \calH_1[f|M_1],
   \] 
   for some constant $N^{(0)}=C_{\Lambda, d}>0$.
\end{lemma}

\begin{proof}
    First, let us recall the following classical relative entropy inequality
    \begin{equation}\label{eq: crei}
        H_1[f|M_1] \ge \frac{|j_f|^2}{2\rho_f} + \rho_f \log \rho_f.
    \end{equation}
    To see this, we decompose along a local Maxwellian which shares the same mass and momentum with $f$. Namely:
    \begin{align*}
        H_1[f|M_1] &= \intr f \log \frac{f}{M^{(\rho_f,u_f,1)}} + \intr f \log \frac{M^{(\rho_f,u_f,1)}}{M_1}, \quad u_f := \frac{j_f}{\rho_f},
    \end{align*}
    where we recall that
    \begin{equation*}
        M^{(\rho_f,u_f,1)} = \frac{\rho_f}{(2\pi)^{d/2}} \exp\lt(-\frac{|v-u_f|^2}{2} \rt).
    \end{equation*}
    On the one hand, thanks to Jensen's inequality, we have
    \begin{equation*}
        \intr f \log \frac{f}{M^{(\rho_f,u_f,1)}} = \intr \frac{f}{M^{(\rho_f,u_f,1)}} \log \frac{f}{M^{(\rho_f,u_f,1)}} \, (M^{(\rho_f,u_f,1)} \dv) \ge 0.
    \end{equation*}
    On the other hand,
    \begin{align*}
        \intr f \log \frac{M^{(\rho_f,u_f,1)}}{M_1} &= \intr f \lt( \log \rho_f + v \cdot u_f - \frac{|u_f|^2}{2} \rt) \\
        &= \rho_f \log \rho_f + \frac{1}{2} \rho_f |u_f|^2 \\
        &= \rho_f \log \rho_f + \frac{|j_f|^2}{2\rho_f},
    \end{align*}
    which proves \eqref{eq: crei}. Note that $\intt \rho_f = 1$ implies $\intt \rho_f \log \rho_f \ge 0$ by Jensen's inequality. 
    Therefore, from \eqref{eq: crei} we deduce
    \begin{equation}\label{eq:peach}
       \intt |j_f|^2\,\dx \le 2\Lambda \calH_1[f|M_1] .
    \end{equation}

    To conclude the bound for $\calC[f]$, we need a bound on $\nabla_x\Phi_f$. By standard energy estimates,
    \begin{equation}\label{eq:blueberry}
    \|\nabla_x \Phi_f\|_{L^2} \le C_d\|\rho_f-1\|_{L^2}.
    \end{equation}
    Applying \eqref{eq:peach}--\eqref{eq:blueberry} and the Cauchy-Schwarz inequality together with Lemma \ref{lem:baby1}, we get
    \[
    |\calC[f]| \le \Lambda (2C_d)^{\frac{1}{2}}\calH_1[f|M_1].
    \]
\end{proof}

We then state the main goal of this section:
\begin{proposition}\label{prop: linFP}
We assume that the solution $f$ to \eqref{eq: linFP} satisfies the bounds in \eqref{paella}. Then for all $t\ge T_0$, it holds that
    \[
    \ddt \calC[f] \le - N^{(1)} \calP_1[\rho_f] + N^{(2)}\calD_1[f]+ N^{(3)}\calH_1[f|M_{1,\rho_f}]
    \]
    for some constants $N^{(1)}=C_\lambda$,  $N^{(2)}=C_{E,d}$, $N^{(3)}=C_{\Lambda, d}$.
\end{proposition}

\subsection{Proof of Proposition \ref{prop: linFP}}

Towards the end of a proof of Proposition \ref{prop: linFP}, we first compute the time derivative of the corrector term. The following lemma is similar to  \cite{APT26}.

\begin{lemma}
    For any solution $f$ to \eqref{eq: linFP}, it holds at least formally that
    \[
    \begin{aligned}
    \ddt \calC[f] &= \intt \left(-(\rho_f-1)^2  + S_f:\nabla_x^2\Phi_f -j_f\cdot \nabla_x \Phi_f - j_f \cdot \nabla_x(-\Delta_x)^{-1} \nabla_x\cdot j_f\right)\,\dx\\
    &=:\calI_1+\calI_2+\calI_3+\calI_4,
    \end{aligned}
    \]
    where 
    \[
    S_f:=\intr v\otimes v (f-M_{1,\rho_f})\,\dv.
    \]
\end{lemma}

\begin{proof}
    By testing $v$ against \eqref{eq: linFP} and exploiting the fact that $\intr v \otimes v M_{1,\rho_f}\,\dv=\rho_f \bbI_d$, we obtain
    \[
    \pa_t j_f = -\nabla_x \cdot \lt(\intr v\otimes v f\,\dv\rt) - j_f = -\nabla_x \rho_f - \nabla_x \cdot S_f - j_f.
    \]
Then we notice that
\[
    \begin{aligned}
    \ddt \calC[f] &= \intt  (-\nabla_x\rho_f - \nabla \cdot S_f -j_f)\cdot \nabla_x \Phi_f  + j_f\cdot \p_t \nabla_x \Phi_f \,\dx
    \end{aligned}
    \]
  On the one hand, the terms $\calI_1$ and $\calI_2$ are easily obtained through the divergence theorem, where we use the fact that $(-\Delta_x)\Phi_f = \rho_f-1$ and $\intt \rho_f\,\dx=1$. On the other hand, $\calI_4$ is obtained by using the fact that $\rho_f$ satisfies the continuity equation $\pa_t \rho_f + \nabla_x \cdot j_f =0$, obtained by integrating \eqref{eq: linFP} over $\R^d_v$. 
\end{proof}

\begin{lemma}[Estimate of $\calI_2$] For any $f$ satisfying the conditional bounds \eqref{paella}, there holds
\[
 \calI_2 \le N^{(2)}\calD_1[f] + \frac{1}{4}\|\rho_f-1\|_{L^2}^2
\] 
for some $N^{(2)} = C_{E,d}>0$.

\end{lemma}

\begin{proof}
    We notice that
\[
S_f = \intr v\otimes v f\,\dv - \lt(\intr f\,\dv\rt)\bbI_d \\
= \intr v \otimes (vf + \nabla_v f)\,\dv,
\]
which implies
\[
|S_f|\le \lt(\intr |v|^2 f\,\dv\rt)^{\frac{1}{2}}(D_1[f])^{\frac{1}{2}} \le E^{\frac{1}{2}}(D_1[f])^{\frac{1}{2}}.
\]
Therefore, Young's inequality shows
\[
\intt S_f:\nabla_x^2\Phi_f\,\dx \le \frac{E}{2\e}\intt D_1[f]\,\dx + \frac{\e}{2}\intt |\nabla_x^2 \Phi_f|^2\,\dx 
\]
for any $\e>0$. Thanks to the Calderon-Zygmund inequality we know that $\|\nabla_x^2 \Phi_f\|_{L^2}\le C_d\|\rho_f-1\|_{L^2}$; hence, choosing $\e = \frac{1}{2C_d^2}$ we obtain
\[
\intt S_f:\nabla_x^2\Phi_f\,\dx \le {EC_d^2}\calD_1[f] + \frac{1}{4}\|\rho_f-1\|_{L^2}^2.
\]
\end{proof}

\begin{lemma}[Estimates of $\calI_3$ and $\calI_4$] \label{lem:lin:i3i4}
    For any $f$ satisfying the conditional bounds \eqref{paella},
    \begin{equation*}
        \calI_3 + \calI_4 \le N^{(3)} \calH_1[f|M_{1,\rho_f}] + \frac{1}{4} \|\rho_f - 1\|_{L^2}^2,
    \end{equation*}
    where $N^{(3)}=C_{\Lambda, d}>0$.
\end{lemma}
\begin{proof}
We estimate $\calI_3$ first. Recall that, in the same way as in the proof of Lemma \ref{lem: m=1 comp}, there holds
\[
\begin{aligned}
H_1[f|M_{1,\rho_f}] &= \intr f \log \frac{f}{M^{(\rho_f,u_f,1)}}\,\dv + \intr f \log \frac{M^{(\rho_f,u_f,1)}}{M^{(\rho_f,0,1)}}
\,\dv \\
&\ge \intr f \log \frac{M^{(\rho_f,u_f,1)}}{M^{(\rho_f,0,1)}}
\,\dv = \frac{|j_f|^2}{2\rho_f}.
\end{aligned}
\]
Since $\rho_f \le \Lambda$, we have
\begin{equation}\label{eq: j2}
|j_f|^2 \le 2\Lambda H_1[f|M_{1,\rho_f}].
\end{equation}
Thus,
\begin{align*}
    \intt j_f\cdot \nabla_x \Phi_f\,\dx \le \|j_f\|_{L^2} \|\nabla_x\Phi_f \|_{L^2} &\le (2\Lambda)^{\frac{1}{2}}C_d\lt(\calH_1[f|M_{1,\rho_f}] \rt)^{\frac{1}{2}}  \|\rho_f - 1\|_{L^2}  \\
    &\le (2\Lambda)C_d^2 \calH_1[f|M_{1,\rho_f}] + \frac{1}{4}\|\rho_f - 1\|_{L^2}^2.
\end{align*}
Next, we estimate $\calI_4$. Next, we can infer that
\begin{equation*}
    |\calI_4| = \lt| \intt j_f \nabla_x(-\Delta_x)^{-1} \nabla_x\cdot j_f \,\dx \rt| \le C_d\|j_f\|_{L^2}^2\\
    \le C_d(2\Lambda)  \calH_1[f|M_{1,\rho_f}],
\end{equation*}
by Calderon-Zygmund type argument, where we used \eqref{eq: j2} in the last inequality.
This completes the proof.
\end{proof}

\

We now conclude the proof of Proposition \ref{prop: linFP}. First, we recall from Lemma \ref{lem:baby1} that
    \[
    \calI_1 = \intt -(\rho_f-1)^2\,\dx \le -\lambda \calP_1[\rho_f].
    \]
We finish by combining the previous three lemmas.

The proof of Theorem \ref{thm: main} for the case $m=1$ is then an application of Proposition \ref{prop: linFP}. Its proof is in fact universal for both $m=1$ and $m>1$. Therefore, we first prove the analogue of Proposition \ref{prop: linFP} for the case $m>1$ in the next section. Then, we present the proof of Theorem \ref{thm: main} in Section \ref{sec: thm}.


\section{The nonlinear case $m>1$} \label{sec: nonlin}
In this section we move on to our main goal of $m>1$ in
\begin{equation}\label{eq: pmFP}
 (\pa_t + v\cdot \nabla_x)f = \Delta_v f^m + \nabla_v \cdot (vf).
\end{equation}

\subsection{Set-up}
We recall the entropy associated to \eqref{eq: pmFP}:
\[
H_m[f]:= \intr \lt(\frac{1}{2}|v|^2 + \frac{f^{m-1}}{m-1}\rt)f\,\dv,
\]
and the dissipation
\[
D_m[f]:= \intr f\lt|\nabla_v \lt( \frac{m}{m-1}f^{m-1} + \frac{|v|^2}{2} \rt)\rt|^2 \,\dv. 
\]
Also, for $m>1$, the relative entropy is now defined as
\[
H_m[f|g]:=H_m[f]-H_m[g].
\]
It is well-known that, up to a translation, the global minimizer of $H_m[f]$ under the mass constraint $\int f \,dv=\rho_f$ is given by $M_{m,\rho_f}$, see \cite{O01,CJMTU01}. In particular, this implies
$H_m[f|M_{m,\rho_f}]\ge 0$.
We also point that by a direct computation, it holds $H_m[M_{m,\rho}]=c_{m,d} \rho^\gamma /\gamma$. Also, for later use, we record that
\begin{equation}\label{iberico}
    \lt(\intr M_{m,\rho_f}^m\,\dv\rt)\bbI_d = (\gamma-1)\frac{c_{m,d}}{\gamma} \rho_f^\gamma \bbI_d = \intr (v\otimes v) M_{m,\rho_f} \,\dv.
\end{equation}

By means of the generalized logarithmic Sobolev inequality \cite[Corollary 13]{DPD02}, the lower bound for the entropy dissipation rate holds:
\begin{equation}\label{eq:LSIm}
 D_m[f] \ge 2H_m[f|M_{m,\rho_f}].
\end{equation}

We now set
\begin{equation*}
    \calH_m[\cdot] := \intt H_m[\cdot] \,\dx, \quad \calD_m[\cdot] := \intt D_m[\cdot] \,\dx,
\end{equation*}
and use the analogous notation $\calH_m[f|g] = \calH_m[f] - \calH_m[g]$. Since $\calH[M_m] =c_{m,d}/\gamma$ is constant, \eqref{eq:LSIm} thus provides
\begin{align*}
\ddt \calH_m[f|M_m] = \ddt \calH_m[f] \le -\calD_m[f]&\leq -\frac12 \calD_m[f]-\calH_m[f|M_{m,\rho_f}] \\&= -\frac12 \calD_m[f]-\calH_m[f|M_m] + \calH_m[M_{m,\rho_f}|M_m].
\end{align*}
It remains to get decay for:
\begin{equation*}
\calH_m[M_{m,\rho_f}|M_m] = \frac{c_{m,d}}{\gamma} \intt  (\rho_f^\gamma -1) \,\dx =: \calP_\gamma[\rho_f]. 
\end{equation*}
Notice that $P_\gamma[\rho_f]\ge0$ thanks to Jensen's inequality since $\intt \rho_f\,\dx=\intt 1\,\dx$.

Because of the constants arising in the Barenblatt profile, it is more cumbersome compared to Section \ref{sec: lin} when it comes down to keeping track of the constants. Therefore, we will merely denote $X\lesssim_a Y$ if $X\le C_a Y$ for some constant $C_a>0$ depending on $a$.

\subsection{The corrector term}
We use the same corrector term \eqref{coriolis} as in the previous section:
\[
\calC[f]:=\intt j_f \cdot \nabla_x \Phi_f\,\dx.
\]

As in Lemma \ref{lem: m=1 comp}, we need to first control the crossing term $\calC[f]$ in terms of $\calH_m[f|M_m]$. First, we establish a result analogous to Lemma \ref{lem:baby1}.
\begin{lemma}\label{lem:babym}
   If $\rho_f \ge 0$ satisfies $\intt \rho_f\,\dx =1$ and $\lambda \le \rho_f \le \Lambda$ for some $0< \lambda \le \Lambda <+\infty$, then it holds that
   \[
    \|\rho_f-1\|_{L^2}^2 \lesssim_{\lambda,\Lambda,\gamma} \calP_\gamma[\rho_f]\lesssim_{\lambda,\Lambda,\gamma}  \|\rho_f-1\|_{L^2}^2 .
   \]
\end{lemma}
\begin{proof}
    Define $\psi_\gamma(z)= z^\gamma-1$ for $z>0$. For $z \in [\lambda,\Lambda]$, we have
    \[
    \psi_\gamma(z) = \gamma(z-1) + {\gamma(\gamma-1)}{\tilde{z}}^{2-\gamma}(z-1)^2
    \]
    for some $\tilde{z} \in [\lambda, \Lambda]$. Putting $z=\rho_f$ and then integrating over $\T^d$, we obtain the desired inequality because $\intt \rho_f\,\dx = 1$.
\end{proof}

We then prove the analogue of Lemma \ref{lem: m=1 comp}:

\begin{lemma}\label{lem: comp}

For any $f$ satisfying the conditional bound \eqref{paella}, it holds that
\[
|\calC[f]| \lesssim_{\Lambda, \gamma, d} \calH_m[f|M_m].
\]
\end{lemma}

\begin{proof}
The Cauchy--Schwarz inequality gives
\[
|\calC[f]| \le \|j_f\|_{L^2}\|\nabla_x \Phi_f\|_{L^2}.
\]
By the Poincar\'e and Calder\'on--Zygmund inequalities, the second term is easily bounded as
\[
\|\nabla_x \Phi_f\|_{L^2} \lesssim_d \|\nabla_x^2 \Phi_f\|_{L^2} \lesssim_d \|\rho_f -1\|_{L^2} \lesssim_{\lambda,\Lambda,\gamma} (\calP_\gamma[\rho_f])^{\frac{1}{2}},
\]
where the last inequality comes from Lemma \ref{lem:babym}.
Then, by recalling that $H[f|M_{m,\rho_f}|\ge 0$, we can estimate $\calP_\gamma[\rho_f]$ as:
\begin{align*}
    \calH_m[f|M_m] &= \calH_m[f|M_{m,\rho_f}] + \calH_m[M_{m,\rho_f}|M_m] \\
    &\ge \calH_m[M_{m,\rho_f}|M_m] =\calP_\gamma[\rho_f].
\end{align*}
On the other hand, to control $\|j_f\|_{L^2}$, we recall the following inequality that comes from optimality conditions of related functionals  \cite{BV05,B99,KS25}:
\begin{equation}\label{eq: min pri}
H_m[f] = \intr \lt(\frac{1}{2}|v|^2 f + \frac{f^m}{m-1} \rt)\,\dv \ge \frac{|j_f|^2}{2\rho_f} + \frac{c_{m,d}}{\gamma} \rho_f^\gamma,
\end{equation}
which holds as long as $\rho_f>0$. Otherwise, the left hand side is zero. This shows that
\[
H_m[f|M_m]\ge \frac{|j_f|^2}{2\rho_f}+\frac{c_{m,d}}{\gamma}(\rho_f^\gamma-1).
\]
Hence, since $\rho_f \le \Lambda$, integrating both sides of \eqref{eq: min pri} over $\T^d$ shows (remembering that $\calP_\gamma[f]\ge 0$)
\begin{equation}\label{eq/j2}
\calH_m[f|M_m] \ge \intt \frac{|j_f|^2}{2\rho_f}\,\dx \gtrsim_{\Lambda, \gamma} \|j_f\|_{L^2}^2. 
\end{equation}
This completes the proof.
\end{proof}

\

Our goal is to prove that this corrector term induces dissipation of $\calP_\gamma[\rho_f]$.

\begin{proposition}\label{prop: corrector}
We assume that $f$ is a  solution to \eqref{eq: pmFP} satisfying the conditional bounds \eqref{paella}. Then it holds for every $t\ge T_0$ that
    \begin{equation*}
        \ddt \calC[f] \le - N^{(1)} \calP_\gamma[\rho_f] + N^{(2)} \calD_m[f] + N^{(3)}\calH_m[f|M_{m,\rho_f}]
    \end{equation*}
    for some constants $N^{(i)}>0$ independent of time.
\end{proposition}

\subsection{Proof of Proposition \ref{prop: corrector}}

In order to achieve this result, we first take the time derivative of the corrector term, obtaining
\begin{lemma}
For any solution $f$ to \eqref{eq: pmFP}, it holds at least formally that
\[
\begin{aligned}
&\ddt \calC[f] \\
&= \intt \lt(-  (\gamma-1) \frac{c_{m,d}}{\gamma}  (\rho_f^\gamma - 1) (\rho_f - 1) + S_{m,f} : \nabla_x^2 \Phi_f - j_f\cdot \nabla_x\Phi_f + j_f \cdot \nabla_x(-\Delta_x)^{-1}\nabla_x\cdot j_f \rt) \,  \dx \\
&=: \calJ_1 + \calJ_2 + \calJ_3 + \calJ_4,
\end{aligned}
\]
where
\begin{equation*}
    S_{m,f}(t,x) := \intr v\otimes v (f - M_{m,\rho_f}) \,\dv .
\end{equation*}
\end{lemma}
\begin{proof}
Using the moment equation
\begin{equation*}
\pa_t j_f = - \nabla_x\cdot \intr v \otimes v f\,\dv -j_f,
\end{equation*}
along with \eqref{iberico}, we obtain
\begin{equation}\label{tbonesteak}
\begin{aligned}
&\ddt \calC[f] \\
&= \intt \lt(- (\gamma-1) \frac{c_{m,d}}{\gamma} \nabla_x \rho_f^\gamma \cdot \nabla_x \Phi_f + S_{m,f} : \nabla_x^2 \Phi_f - j_f\cdot \nabla_x\Phi_f + j_f \cdot \pa_t \nabla_x \Phi_f \rt) \,  \dx .
\end{aligned}
\end{equation}
The first term of \eqref{tbonesteak} can be rewritten as
\begin{align*}
    & -  (\gamma-1) \frac{c_{m,d}}{\gamma} \int_{\T^d} \nabla_x (\rho_f^\gamma - 1) \cdot \nabla_x \Phi_f \,\dx \\
    &=  (\gamma-1) \frac{c_{m,d}}{\gamma} \int_{\T^d} (\rho_f^\gamma - 1) \Delta_x (-\Delta_x)^{-1} (\rho_f - 1) \,\dx \\
    &= -  (\gamma-1) \frac{c_{m,d}}{\gamma} \int_{\T^d} (\rho_f^\gamma - 1) (\rho_f - 1) \,\dx,
\end{align*}
by using the mean-value theorem to obtain the last line. For the last term $j_f\cdot \p_t \nabla_x \Phi_f$ in \eqref{tbonesteak}, we use the continuity equation $\p_t \rho_f + \nabla_x\cdot j_f$ to obtain the form of $\calJ_4$.
\end{proof}

Notice that if $f$ obeys macroscopic bounds \eqref{paella}, then it is easy to see that
\begin{equation}\label{eq / Cstar}
\calJ_1 \le - C_\star \|\rho_f - 1\|_{L^2}^2.
\end{equation}
for some constant $C_\star>0$ depending on $\gamma,\lambda,m,d$.

\begin{lemma}[Estimate of $\calJ_2$] \label{lem:J2}
    We assume that $f$ satisfies the conditional bounds \eqref{paella}. Then for any $\eta>0$, there exists $C=C_{\eta,m,d,E,\Lambda_m}>0$ such that
    \begin{equation*}
        \calJ_2 \le C \lt( \calD_m[f] + \calH_m[f|M_{m,\rho_f}] \rt) + \eta \|\rho_f - 1\|_{L^2}^2.
    \end{equation*}
\end{lemma}
\begin{proof}
We establish pointwise estimates for $S_{m,f}$. Invoking the identities in \eqref{iberico}, we get
\[
\begin{aligned}
S_{m,f} = \intr v \otimes v (f - M_{m,\rho_f})\,\dv &= \intr v\otimes v f\,\dv - \lt(\intr M_{m,\rho_f}^m \,\dv \rt) \bbI_d   \\
&= \intr \Big[ v\otimes v f - f^m \bbI_d \Big] \,\dv + \lt(\intr f^m- M_{m,\rho_f}^m\,\dv\rt)\bbI_d \\
&= -\intr (\nabla_v f^m + vf)\otimes v \,\dv + \lt(\intr f^m- M_{m,\rho_f}^m\,\dv\rt)\bbI_d \\
&=:S_{m,f}^1 + S_{m,f}^2.
\end{aligned}
\]
By applying the Cauchy-Schwarz inequality, the first term can be controlled as
\[
\begin{aligned}
|S_{m,f}^1|=\lt|\intr (\nabla_v f^m + vf)\otimes v \,\dv\rt| &\le \lt(\intr |v|^2f\,\dv\rt)^{\frac{1}{2}}\lt(\intr \frac{|\nabla_v f^m + vf|^2}{f}\,\dv\rt)^{\frac12} \\
&\lesssim_E {D_m[f]}^{\frac12},
\end{aligned}
\]
where the last inequality follows from the macroscopic bounds \eqref{paella}.
To handle the second term $S_{m,f}^2$, we use the following inequality (see Section \ref{app} for its proof)
\begin{equation}\label{carbonara}
\intr |f^m - M_{m,\rho_f}^m|\,\dv \le 
C_m H_m[f|M_{m,\rho_f}] + C_{m,d} H_m[f|M_{m,\rho_f}]^{\frac{1}{2}} . 
\end{equation}
Using that $H_m[f|M_{m,\rho_f}]\le H_m[f]$, along with \eqref{paella}, we deduce from \eqref{carbonara}
\begin{equation*}
|S_{m,f}^2| \le \intr |f^m - M_{m,\rho_f}^m|\,\dv \lesssim_{m,d,E,\Lambda_m} H_m[f|M_{m,\rho_f}]^{\frac{1}{2}}.
\end{equation*}
Hence,
\[
\begin{aligned}
|\calJ_2| &\lesssim_{m,E,\Lambda_m} \lt\| D_m[f]^{\frac{1}{2}} + H_m[f|M_{m,\rho_f}]^{\frac{1}{2}} \rt\|_{L^2} \|\nabla_x^2 \Phi_f\|_{L^2} \\
&\lesssim_d (\calD_m[f] + \calH_m[f|M_{m,\rho_f}])^{\frac{1}{2}}  \|\rho_f - 1\|_{L^2} \\
&\le C_\eta (\calD_m[f] + \calH_m[f|M_{m,\rho_f}]) + \eta \|\rho_f - 1\|_{L^2}^2 \quad \forall \eta > 0,
\end{aligned}
\]
where we used the Cauchy--Schwarz, Calder\'on--Zygmund, and Young's inequalities.
\end{proof}

\begin{lemma}[Estimates of $\calJ_3$ and $\calJ_4$]
    We assume that $f$ satisfies the conditional bounds \eqref{paella}. Then for each $\eta>0$, there exists $C=C_{\Lambda,d}>0$ such that
    \begin{equation*}
        \calJ_3 + \calJ_4 \le C\calH_m[f|M_{m,\rho_f}] + \eta \|\rho_f - 1\|_{L^2}^2.
    \end{equation*}
\end{lemma}

\begin{proof}
    By invoking \eqref{eq/j2}, this follows in exactly the same way as with Lemma \ref{lem:lin:i3i4}. 
\end{proof}

\begin{proof}[Proof of Proposition \ref{prop: corrector}]
We now fix $\eta:= C_\star/4>0$ where $C_\star$ is defined in \eqref{eq / Cstar}. Collecting the previous three lemmas, we then obtain that 
\begin{align*}
    \ddt \calC[f] \le -\frac{C_\star}{2} \|\rho_f - 1\|_{L^2}^2 + N^{(2)} \calD_m[f] + N^{(3)} \calH_m[f|M_{m,\rho_f}],
\end{align*}
for some constants $N^{(2)},N^{(3)}$ depending on $\gamma, m, d$ and the conditional bounds \eqref{paella}.
We then invoke Lemma \ref{lem:babym} to conclude to the assertion of Proposition \ref{prop: corrector}.
\end{proof}




\subsection{Proof of \eqref{carbonara}} \label{app}

Here, we prove the technical inequality \eqref{carbonara} used in the proof of Lemma \ref{lem:J2}. We define $\Psi_m(z)=z^m/(m-1)$ for $z \in \R_+$, and then an entropy and a relative entropy with respect to $\Psi_m$:
\[
E_m[f]:=\intr \Psi_m(f)\,\dv,\quad E_m[f|g]:= \intr \lt( \Psi_m(f) - \Psi_m(g) - \Psi_m'(g)(f-g) \rt) \,\dv.
\]
We then collect some inequalities: First, from \cite[Proposition 5]{CJMTU01}, we know that for any $f \in L^1_2\cap L^m_+(\R^n)$, it holds that
\begin{equation}\label{eq:elementary}
H_m[f|M_{m,\rho_f}] \ge E_m[f|M_{m,\rho_f}].
\end{equation}
Second, we state here a generalized Csisz\'ar--Kullback inequality. Albeit this has been well-known (see \cite{CT00}, \cite[Section 6]{CJMTU01}, \cite[Chapter 3]{Matthes}, and \cite{ACM26} for the most recent improvement), we believe the literature results have addressed the case of unit mass. Here we treat the case of general mass to exhibit dependence on mass by simply resorting to a scaling argument.

\begin{lemma}[Generalized Csisz\'ar--Kullback inequality]\label{tianyuan}
    For any $f \in L^1_2\cap L^m_+(\R^d)$,
\begin{equation*}
    \|f-M_{m,\rho_f}\|_{L^1} \le  C_{m,d} \rho_f^{-(\gamma-1)} H_m[f|M_{m,\rho_f}]^{\frac{1}{2}},
\end{equation*}
where $\gamma>1$ is defined as in \eqref{lmm}.
\end{lemma}
\begin{proof}
    Let us first define an appropriate scaling on $L^1_2 \cap L^m_+(\R^d)$, based upon the relations satisfied by Barenblatt profiles. Writing
    \begin{equation*}
        f_*(v) := \rho_f^{-a} f\lt( \rho_f^b v \rt),
    \end{equation*}
    we observe that with the choice of $b = \frac{\gamma-1}{2}$, there holds
    \begin{align*}
        (M_{m,\rho_f})_*(v) &= \rho_f^{-a} \lt( \frac{m-1}{m} \lt( c_{m,d} \rho_f^{\gamma-1} - \frac{| \rho_f^{b} v|^2}{2} \rt) \rt)_+^{\frac{1}{m-1}} \\
        &= \rho_f^{-a} \rho_f^{\frac{\gamma-1}{m-1}} \lt(\frac{m-1}{m} \lt( c_{m,d} - \frac{|v|^2}{2}\rt) \rt)_+^{\frac{1}{m-1}} \\
        &= \rho_f^{-a + \frac{\gamma-1}{m-1}} M_m(v).
    \end{align*}
    Thus, choosing $a=\frac{\gamma-1}{m-1}$, we obtain that
    \begin{equation*}
        (M_{m,\rho_f})_*(v) = M_m(v).
    \end{equation*}
    
    With these choices of $a,b$, we note that for any $f\in L^1_2 \cap L^m_+(\R^d)$, there holds
    \begin{equation*}
        \intr f_*(v) \,\dv = 1,
    \end{equation*}
    as well as
    \begin{equation}\label{boon soo}
        \frac{\iinttr f_*^m\,\dx\dv}{\iinttr f^m\,\dx\dv} = \frac{\iinttr |v|^2 f_*\,\dx\dv}{\iinttr |v|^2 f \,\dx\dv}.
    \end{equation}
    We further remark that \eqref{boon soo} guarantees that the entropies of $f_*$ and $f$ are comparable. More specifically:
    \begin{align*}
        H_m[f_*] = \rho_f^{-\frac{d(m-1)+2m}{d(m-1)+2}} H_m[f].
    \end{align*}

    Consequently, by rescaling the Csisz\'ar--Kullback inequality known to hold for unit mass functions, we obtain for some absolute constant $C_{m,d}>0$ that
    \begin{equation*}
    \begin{split}
        \|f - M_{m,\rho_f}\|_{L^1} = \rho_f \|f_* - M_m\|_{L^1} \le C_{m,d} \rho_f H_m[f_*|M_m] &= C_{m,d} \rho_f^{1-\frac{d(m-1)+2m}{d(m-1)+2}} H_m[f|M_{m,\rho_f}] \\
        &= C_{m,d} \rho_f^{-(\gamma-1)} H_m[f|M_{m,\rho_f}].
    \end{split}
    \end{equation*}
\end{proof}

We now prove \eqref{carbonara}.

\begin{lemma}
Let $m > 1$. For any $f \in L^1_2\cap L^m_+(\R^d)$ such that $f^{m-\frac{1}{2}} \in \dot{H}^1(\R^d)$,
\[
\intr |f^m -M_{m,\rho_f}^m|\,\dv \le   (m-1)H_m[f|M_{m,\rho_f}] + C_{m,d}  H_m[f|M_{m,\rho_f}]^{\frac{1}{2}}.
\]
\end{lemma}
\begin{proof}
We first notice that
\[
\begin{aligned}
&f^m -M_{m,\rho_f}^m \\
&= (m-1)\lt(\Psi_m(f)-\Psi_m(M_{m,\rho_f})- \Psi_m'(M_{m,\rho_f})(f-M_{m,\rho_f})\rt) + mM_{m,\rho_f}^{m-1}(f-M_{m,\rho_f}),
\end{aligned}
\]
so that
\begin{equation}\label{chicago pizza}
\intr |f^m -M_{m,\rho_f}^m|\,\dv \le (m-1) E_m[f|M_{m,\rho_f}]+\intr mM_{m,\rho_f}^{m-1}|f-M_{m,\rho_f}|\,\dv.
\end{equation}
We first recall \eqref{eq:elementary} to see $(m-1)E_m[f|M_m]\le (m-1)H_m[f|M_m]$. For the second term on the right-hand side of \eqref{chicago pizza},  the definition of $M_{m,\rho_f}$ in \eqref{lmm} leads us to note that $mM_{m,\rho_f}^{m-1} \le (m-1)c_{m,d}\rho_f^{\gamma-1}$ holds pointwise. Thus, we can combine this with Lemma \ref{tianyuan} to estimate
\[
\intr mM_{m,\rho_f}^{m-1}|f-M_{m,\rho_f}|\,\dv \lesssim_{m,d} H_m[f|M_{m,\rho_f}]^{\frac{1}{2}}.
\]
Collecting all, we conclude to \eqref{carbonara}. 
\end{proof}

\section{Proof of  Theorem \ref{thm: main}} \label{sec: thm}
We now present the proof of Theorem \ref{thm: main} in a framework uniform in $m\in [1,\infty)$. We assume that $t \ge T_0$ so that the conditional bounds \eqref{paella} are valid. We apply the results of Proposition \ref{prop: linFP} or Proposition \ref{prop: corrector}, depending on whether $m=1$ or $m>1$. Due to the entropy inequality and logarithmic Sobolev inequalities \eqref{eq:LSI1}--\eqref{eq:LSIm}, we have
\[
\ddt \calH_m[f|M_m] = \ddt \calH_m[f] \le -\frac{1}{2}\calD_m[f] -\calH_m[f|M_{m,\rho_f}].
\]
We now choose small $\delta>0$ satisfying $\max\{N^{(0)
},N^{(2)
},N^{(3)
}\}\delta \le \frac{1}{2}$. Thanks to either Lemma \ref{lem: m=1 comp} or Lemma \ref{lem: comp}, we have in this way that
\begin{equation}\label{friends}
\frac{1}{2}\calH_m[f|M_m]\le \calH_m[f|M_m]+\delta \calC[f]\le \frac{3}{2}\calH_m[f|M_m].
\end{equation}
On the other hand,
\[
\begin{aligned}
&\quad \ddt \lt(\calH_m[f|M_1] + \delta\calC[f]\rt) \\ &\le  -\frac{1}{2}\calD_m[f] -\calH_m[f|M_{m,\rho_f}] +\delta\lt( - N^{(1)} \calP_\gamma[\rho_f] + N^{(2)} \calD_m[f] + N^{(3)} \calH_m[f|M_{m,\rho_f}]\rt)  \\
&\le -N^{(1)} \delta \calP_\gamma[\rho_f] - \frac{1}{2}\calH_m[f|M_{m,\rho_f}] \le -\min\lt\{N^{(1)}\delta,\frac{1}{2}\rt\}\calH_m[f|M_m] \\
&\le -\mu\lt(\calH_m[f|M_m] + \delta\calC[f]\rt),
\end{aligned}
\]
where $\mu:=\frac{2}{3}\min\lt\{N^{(1)}\delta,\frac{1}{2}\rt\}>0$.
This proves
\[
\begin{aligned}
\lt(\calH_m[f|M_m] + \delta\calC[f]\rt)(t) &\le \lt(\calH_m[f|M_m] + \delta\calC[f]\rt)(T_0)e^{-\mu(t-T_0)}.
\end{aligned}
\]
Thanks to \eqref{friends},
\[
\calH_m[f(t)|M_m] \le 3\calH_m[f(T_0)|M_m]e^{-\mu (t-T_0)}.
\]
Finally, the Csisz\'ar--Kullback inequality yields (see Remark \ref{rem / csis} below)
\[
\|f(t)-M_m\|_{L^1(\T^d\times \R^d)} \le C_{m,d}\calH_m[f(T_0)|M_m]^{\frac{1}{2}}e^{-\frac{\mu}{2}(t-T_0)}.
\]

\begin{remark} \label{rem / csis}
    The Csisz\'ar--Kullback inequality in the kinetic setting, that is, the fact that for every sufficiently integrable $f$ with $\|f\|_{L^1(\T^d\times\R^d)} = 1$ there holds
    \begin{equation*}
        \|f - M_m\|_{L^1(\T^d\times\R^d)} \le C_{m,d} \calH_m[f|M_m]^{1/2},
    \end{equation*}
    is not explicitly stated in the literature. However, we remark that it can be obtained straightforwardly by repeating the proofs of \cite{CT00,CJMTU01,Matthes}, by replacing all of the integrals there with the double integral over $\T^d\times\R^d$.
\end{remark}

\bibliographystyle{siam}
\bibliography{f_m}

\end{document}